\documentclass[10pt]{article}
\usepackage{amssymb,amsmath,amsthm,verbatim}

\newtheorem*{remark}{Remark}
\newtheorem{claim}{Claim}
\newtheorem{theorem}{Theorem}
\newtheorem*{theorem'}{Theorem}
\newtheorem*{theorem-main}{Theorem 4}
\newtheorem{lemma}[theorem]{Lemma}

\newtheorem{conjecture}[theorem]{Conjecture}
\newtheorem{proposition}[theorem]{Proposition}
\newtheorem{corollary}[theorem]{Corollary}

\theoremstyle{remark}
\newtheorem{example}{Example}

\DeclareMathOperator{\Var}{Var}
\newcommand{\val}{val}
\DeclareMathOperator{\Inf}{Inf}
\DeclareMathOperator{\parity}{parity}
\newcommand{\E}{\mathbb E}
\begin{document}
\title{Juntas in the $\ell^1$-grid and Lipschitz maps between discrete tori}

\author{Itai Benjamini\footnote{Weizmann Institute of Science, Israel. e-mail: itai.benjamini@weizmann.ac.il} \\
David Ellis\footnote{Queen Mary, University of London, UK. Research supported in part by a Feinberg Visiting Fellowship from the Weizmann Institute of Science, Israel. e-mail: D.Ellis@qmul.ac.uk} \\
Ehud Friedgut\footnote{Weizmann Institute of Science, Israel. Research supported in part by I.S.F. grant 0398246, and B.S.F. grant 2010247.
e-mail: ehud.friedgut@weizmann.ac.il} \\
Nathan Keller\footnote{Bar Ilan University, Israel. Research supported in part by I.S.F. grant 402/13 and by the Alon fellowship.
e-mail: nkeller@math.biu.ac.il} \\
Arnab Sen\footnote{University of Minnesota, USA.  Research supported in part by N.S.F. grant DMS-1406247. e-mail: arnab@umn.edu}}

\date{August 2015}
\maketitle

\begin{abstract}
We show that if $A \subset [k]^n$, then $A$ is $\epsilon$-close to a junta depending upon at most $\exp(O(|\partial A|/(k^{n-1}\epsilon)))$ coordinates, where $\partial A$ denotes the edge-boundary of $A$ in the $\ell^1$-grid. This bound is sharp up to the value of the absolute constant in the exponent. This result can be seen as a generalisation of the Junta theorem for the discrete cube, from \cite{friedgut-junta}, or as a characterisation of large subsets of the $\ell^1$-grid whose edge-boundary is small. We use it to prove a result on the structure of Lipschitz functions between two discrete tori; this can be seen as a discrete, quantitative analogue of a recent result of Austin \cite{austin}. We also prove a refined version of our junta theorem, which is sharp in a wider range of cases.

\end{abstract}

\section{Introduction}

For $k \in \mathbb{N}$, we write $[k]:= \{1,2,\ldots, k\}$. We work on the {\em $\ell^1$-grid}, the graph $G_{k,n}$ with vertex-set $V(G_{k,n}) = [k]^n$, and edge-set
$$E(G_{k,n}) = \{\{x,y\} \in ([k]^n)^{2}:\ \exists j \in [n]:\ |y_j - x_j| = 1,\ x_i = y_i\ \forall i \neq j\}.$$
In other words, two vectors in $[k]^n$ are joined by an edge if and only if they differ in exactly one coordinate, and their values in this coordinate differ by exactly 1. Note that $G_{2,n}$ is the graph of the $n$-dimensional discrete cube, often denoted by $Q_n$.

If $A \subset [k]^n$, we write $\partial A$ for the {\em edge-boundary} of $A$ in the grid, meaning the set of edges of the grid which join a point in $A$ to a point in $A^c := [k]^n \setminus A$.

Bollob\'as and Leader \cite{grid-iso} proved the following edge-isoperimetric inequality for the $\ell^1$-grid.
\begin{theorem}[Bollob\'as, Leader]
\label{thm:grid-iso}
Let $A \subset [k]^n$ with $|A| \leq k^n/2$. Then
$$|\partial A| \geq \min\{|A|^{1-1/r} rk^{(n/r)-1}:\ r \in \{1,2,\ldots,n\}\}.$$
\end{theorem}
This is essentially best possible; it is sharp (for example) when $|A| = a^s k^{n-s}$ for integers $a,s$ with $1 \leq a \leq k-1$ and $1 \leq s \leq n$, as can be seen by taking $A$ to be the cuboid $[a]^s \times [k]^{n-s}$. (Note that since $\partial(A^c) = \partial A$, Theorem \ref{thm:grid-iso} immediately implies an isoperimetric inequality for subsets $A \subset [k]^n$ with $|A| \geq k^n/2$.)

When applied to large subsets, Theorem \ref{thm:grid-iso} implies the following.
\begin{corollary}[Bollob\'as, Leader]
\label{corr:grid-iso}
Let $A \subset [k]^n$ with $k^n/4 \leq |A| \leq 3k^n/4$. Then
$$|\partial A| \geq k^{n-1}.$$
\end{corollary}
Observe that Corollary \ref{corr:grid-iso} is sharp whenever $|A| = bk^{n-1}$, for some integer $b$ with $k/4 \leq b \leq 3k/4$.

In this paper, we will obtain a description of subsets of $[k]^n$ whose edge-boundary has size within a constant factor of $k^{n-1}$. This description will be sharp (up to absolute constant factors) for subsets of measure bounded away from 0 and 1.

\subsection{A junta theorem for the $\ell^1$-grid}

We proceed to outline some more notation and background. If $A \subset [k]^n$, we define its {\em characteristic function} $1_A$ to be the Boolean function on $[k]^n$ with $1_A(x) = 1$ if $x \in A$ and $1_A(x)=0$ if $x \notin A$. 
We let $\mathbb{R}[[k]^n]$ denote the vector space of all real-valued functions on $[k]^n$. We write $||\cdot||_1$ for the $L^1$-norm on $\mathbb{R}[[k]^n]$, defined by
$$||f||_1 = \frac{1}{k^n} \sum_{x \in [k]^n} |f(x)|\quad (f \in \mathbb{R}[[k]^n].)$$
If $f,g:[k]^n \to \mathbb R$, we say that $f$ and $g$ are {\em $\epsilon$-close} if $||f-g||_1 \leq \epsilon$.

If $g: [k]^n \to \{0,1\}$, we say that $g$ is an {\em M-junta} if there exists a set of coordinates $J \subset [n]$ with $|J| = M$, such that $g(x)$ depends only upon the values $(x_j)_{j \in J}$.

If $f:\{0,1\}^n \to \{0,1\}$, we define the {\em influence of coordinate $j$ on $f$} to be
$$\textrm{Pr}_{x \in \{0,1\}^n} \{f(x) \neq f(x \oplus e_j)\},$$
where $\textrm{Pr}_{x \in \{0,1\}^n}$ denotes the probability when $x$ is chosen uniformly at random from $\{0,1\}^n$, $\oplus$ denotes modulo 2 addition, and $e_j$ denotes the $j$th unit vector, i.e. the vector with 1 in the $j$th coordinate and zeros elsewhere. We define the {\em total influence} of $f$ to be the sum of its influences,
$$\Inf(f) = \sum_{j=1}^{n} \Inf_j(f).$$
Finally, if $\Omega$ is any set, $j \in [n]$, and $x \in \Omega^n$, we define the {\em fibre at $x$ in direction $j$} to be the set
$$a_j^x = \{y \in \Omega^n:\ y_i = x_i\ \forall i \neq j\}$$
(i.e., the $j$-coordinate varies freely and all the other coordinates are fixed equal to their values in $x$). Of course, we may identify $a_j^x$ with $[k]$ in a natural way, via the bijection $y_j \mapsto j$. Since $a_j^x$ is independent of $x_j$, by abuse of notation we will sometimes write it as $a_j^y$, where $y \in [k]^{n-1}$ is obtained from $x$ by deleting the $j$th entry of $x$.

Notice that if $f:\{0,1\}^n \to \{0,1\}$, then the influence of coordinate $j$ on $f$ is precisely the probability that $f$ is non-constant on a uniform random fibre in direction $j$.

We will make use of the Junta theorem from  \cite{friedgut-junta}, which characterizes the Boolean functions on $\{0,1\}^n$ of bounded total influence.

\begin{theorem}[Junta theorem, \cite{friedgut-junta}]
\label{thm:friedgut-junta}
Let $f: \{0,1\}^n \to \{0,1\}$. Then for every $\epsilon >0$, there exists an $M$-junta $g:\{0,1\}^n \to \{0,1\}$ such that $f$ is $\epsilon$-close to $g$, and
 $$M \leq \exp(C_0 \Inf(f)/\epsilon),$$
 for some absolute constant $C_0$. (In fact, one can take $C_0 = 2+ \sqrt{3 \ln 2}$.)
\end{theorem}

Note that if $A \subset \{0,1\}^n$, and $f = 1_A$, then $\Inf(f) = |\partial A|/2^{n-1}$, so the above theorem can be restated as follows.

\begin{theorem'}
Let $A \subset \{0,1\}^n$. Then for every $\epsilon >0$, there exists $B \subset \{0,1\}^n$ such that $|A \Delta B| \leq \epsilon 2^n$, and $B$ is a union of subcubes which all have fixed-coordinate set $J$, for some $J \subset [n]$ with
$$|J| \leq \exp(C_0 |\partial A|/(2^{n-1}\epsilon)),$$
where $C_0$ is an absolute constant.
\end{theorem'}
(Recall that a {\em subcube} of $\{0,1\}^n$ is a set of the form
$$\{x \in \{0,1\}^n:\ x_{i} = a_i\ \forall i \in I\}$$
where $I \subset [n]$ and $a_i \in \{0,1\}$ for all $i \in I$. The coordinates in $I$ are called the {\em fixed coordinates} of the subcube, and the rest are called the {\em moving coordinates}. The {\em dimension} of the subcube is $n-|I|$, the number of moving coordinates.)

In this paper, we prove the following generalisation of Theorem \ref{thm:friedgut-junta} for the $\ell^1$-grid.
\begin{theorem}
\label{thm:main}
Let $A \subset [k]^n$. Then for any $\epsilon >0$, $1_A$ is $\epsilon$-close to some $M$-junta $g:[k]^n \to \{0,1\}$ , where
$$M \leq \exp(C_1|\partial A|/(k^{n-1}\epsilon))$$
for some absolute constant $C_1$.
\end{theorem}

This immediately implies an analogous result for the torus. The {\em torus} $T_{k,n}$ is the graph with vertex-set $\mathbb{Z}_k^n$ and edge-set
$$\{\{x,y\} \in (\mathbb{Z}_k^n)^{2}:\ \exists j \in [n]:\ |y_j - x_j|' = 1,\ x_i = y_i\ \forall i \neq j\},$$
where $|s-t|'$ denotes the cyclic distance from $s$ to $t$ in $\mathbb{Z}_k$. If $A \subset \mathbb{Z}_k^n$, we let $\partial'(A)$ denote the edge-boundary of $A$ in the torus, and we let $\partial A$ denote the edge-boundary of $A$ in the grid $G_{k,n}$ (identifying $\mathbb{Z}_k$ with $[k]$ in the natural way). Note that since $G_{k,n}$ is a subgraph of $T_{k,n}$, if $A \subset \mathbb{Z}_k^n$, then $|\partial A| \leq |\partial' A|$. The following is therefore immediate from Theorem \ref{thm:main}.

\begin{corollary}
\label{corr:torus-junta}
Let $A \subset \mathbb{Z}_k^n$. Then for any $\epsilon >0$, $1_A$ is $\epsilon$-close to some $M$-junta $g:\mathbb{Z}_k^n \to \{0,1\}$ , where
$$M \leq \exp(C_1|\partial'(A)|/(k^{n-1}\epsilon))$$
for some absolute constant $C_1$.
\end{corollary}

\subsection{Lipschitz maps between discrete tori}

We will use Corollary \ref{corr:torus-junta} to prove a structure theorem for Lipschitz maps between two discrete tori. To state it, we need a little more terminology. If $x \in \mathbb{Z}_k^n$, we define its {\em $L^1$-norm} by
$$||x||_1 = \frac{1}{n} \sum_{i=1}^n |x_i|',$$
where $|\cdot|'$ denotes cyclic distance from $0$ in $\mathbb{Z}_k$. We say a function $f: \mathbb{Z}_k^{n} \to \mathbb{Z}_l^{m}$ is $\alpha$-{\em Lipschitz with respect to the $L^1$-norm} if
$$||f(x)-f(y)||_1 \leq \alpha ||x-y||_1\quad \forall x,y \in \mathbb{Z}_k^n.$$
We prove the following.

\begin{theorem}
\label{thm:lipschitz-tori}
Suppose $f = (f_1,\ldots,f_m): \mathbb{Z}_k^{n} \to \mathbb{Z}_l^{m}$ is $\alpha$-Lipschitz with respect to the $L^1$-norm. Then for any $\delta,\epsilon >0$, there are at least $(1-\delta)m$ coordinates $i \in [m]$ such that $f_i$ is $\epsilon$-close to some $M$-junta $g_i: \mathbb{Z}_k^n \to \mathbb{Z}_l$, where
$$M \leq l\exp(C_2 \alpha k / (\delta \epsilon)),$$
and $C_2$ is an absolute constant. (In fact, one can take $C_2 = 4C_1$, where $C_1$ is the constant from Theorem \ref{thm:main}.)
\end{theorem}

This can be seen as a discrete, quantitative analogue of the following structure theorem of Austin, concerning Lipschitz maps between two `solid cubes'. If $x \in [0,1]^n$, we define its {\em $L^1$-norm} by
$$||x||_1 = \frac{1}{n} \sum_{i=1}^{n} |x_i|.$$
We say that a function $f:[0,1]^n \to [0,1]^m$ is $\alpha$-{\em Lipschitz with respect to the $L^1$-norm} if
$$||f(x)-f(y)||_1 \leq \alpha||x-y||_1\quad \forall x,y \in [0,1]^n.$$
In \cite{austin}, Austin proves the following.

\begin{theorem}[Austin]
\label{thm:austin}
Let $m,n \in \mathbb{N}$ with $m/n = \gamma$. Then for every $\epsilon >0$ and every $C >0$, there exists an integer $q$ depending only upon $\alpha$, $\gamma$ and $\epsilon$, with the following property. If a function $f: [0,1]^n \to [0,1]^m$ is $\alpha$-Lipschitz with respect to the $L^1$-norm, then there exists another $\alpha$-Lipschitz function $g = (g_1,\ldots,g_m):\ [0,1]^n \to [0,1]^m$, such that
$$\int_{[0,1]^n}||f(x)-g(x)||_1 \mathrm{d}x < \epsilon,$$
and such that each $g_i$ depends upon at most $q$ coordinates in $[n]$.
\end{theorem}

We believe it is slightly more natural to consider Lipschitz functions between discrete tori, rather than between discrete grids, so we have chosen to focus on the former, but in fact, we are also able to prove an analogue of Austin's theorem for Lipschitz functions between discrete $\ell^1$-grids. Naturally, if $f:[k]^n \to [l]^m$, we say that $f$ is {\em $\alpha$-Lipschitz with respect to the $L^1$-norm} if
$$||f(x)-f(y)||_1 \leq \alpha ||x-y||_1\quad \forall x,y \in [k]^n.$$
We prove the following.

\begin{theorem}
\label{thm:lipschitz-grid}
Suppose $f = (f_1,\ldots,f_m): [k]^{n} \to [l]^{m}$ is $\alpha$-Lipschitz with respect to the $L^1$-norm. Then for any $\delta,\epsilon >0$, there are at least $(1-\delta)m$ coordinates $i \in [m]$ such that $f_i$ is $\epsilon$-close to some $M$-junta $g_i: [k]^n \to [l]$, where
$$M \leq (l-1)\exp(C_3 \alpha k / (\delta \epsilon)),$$
and $C_3$ is an absolute constant. (In fact, one can take $C_3 = 2C_1$, where $C_1$ is the constant from Theorem \ref{thm:main}.)
\end{theorem}

Note that the junta-sizes in Theorems \ref{thm:lipschitz-tori} and \ref{thm:lipschitz-grid} have no dependence upon $\gamma = m/n$, in contrast to the situation in Theorem \ref{thm:austin}.

\subsection{Related work}

Generalizations of Theorem~\ref{thm:friedgut-junta} to more general product spaces, and in particular, to the $\ell^1$-grid, have been studied in several previous papers. The first of them is~\cite{fd}, in which the third author (jointly with Dinur) proved a tight generalization to {\it monotone} functions on the $\ell^1$-grid, and conjectured that the same bound holds in the general (non-monotone) case. In~\cite{hatami}, Hatami disproved this conjecture and proved an alternative generalization of the Junta theorem to the $\ell^1$-grid. In~\cite{keller}, the fourth author proved a refinement of the Junta theorem for monotone subsets of the (solid) cube $[0,1]^n$, which can easily be `discretised' to obtain an analogous statement for subsets of the $\ell^1$-grid. In all of these works, the upper bound on the junta-size is governed by the average of a certain quantity $h$ over all fibres, where the value of $h$ on a fibre depends only upon the measure of the set restricted to this fibre. (Recall that a {\em fibre} is a set in which one coordinate varies freely, and the rest are fixed.) As these results are not directly related to our work, and require some more terminology, we delay their exact description until Section~\ref{sec:refinement}.

The prior work perhaps most closely related to ours is a result of Sachdeva and Tulsiani~\cite{s-t}, who proved a junta theorem for general weak products of graphs. In~\cite{s-t}, the bound on the junta-size is given in terms of the edge-boundary of the set.

To state it, we need some additional terminology. If $G = (V,E)$ is a graph, let $G^{\square n}$ denote the {\em $n$-fold weak product} of $G$, the graph with vertex-set $V^n$ and edge-set
$$E(G^{\square n}) = \{\{x,y\} \in (V^n)^{2}:\ \exists j \in [n]:\ \{x_j,y_j\} \in E(G),\ x_i = y_i\ \forall i \neq j\}.$$
Sachdeva and Tulsiani prove the following.

\begin{theorem}[Sachdeva and Tulsiani]
\label{thm:s-t}
Let $G = (V,E)$ be a finite $d$-regular graph on $k$ vertices with log-Sobolev constant $\rho(G)$. Let $A \subset V^n$, and let $\partial A$ denote the edge-boundary of $A$ in the weak product $G^{\square n}$. Then for any $\epsilon >0$, $1_A$ is $\epsilon$-close to some $M$-junta $g: V^n \to \{0,1\}$, where
$$M \leq \exp\left(\frac{C_2}{\rho(G)} \frac{|\partial A|}{dk^{n} \epsilon}\right),$$
for some absolute constant $C_2$.
\end{theorem}

(In fact, Sachdeva and Tulsiani prove a generalisation of this result, for reversible Markov chains on a finite state space.) Since the torus $T_{k,n}$ is precisely the $n$-fold weak product of the $k$-cycle $C_k$, which has log-Sobolev constant $\rho(C_k) = \Theta(1/k^2)$, Theorem \ref{thm:s-t} implies the following.

\begin{theorem'}
Let $A \subset \mathbb{Z}_k^n$. Then for any $\epsilon >0$, $1_A$ is $\epsilon$-close to some $M$-junta $g:\mathbb{Z}_k^n \to \{0,1\}$ , where
$$M \leq \exp(C_3|\partial'(A)|/(k^{n-2}\epsilon))$$
for some absolute constant $C_3$.
\end{theorem'}

Our Corollary \ref{corr:torus-junta} improves this by a factor of $k$ in the exponent.

\subsection{A refined junta theorem}
In Section~\ref{sec:refinement}, we prove a refinement of our junta theorem for the $\ell^1$-grid. To state it, we need some more definitions. If $S \subset [k]$, we let
$$\nu(S) = \frac{|S|}{k}\left(1-\frac{|S|}{k}\right),$$
and we let $\partial S$ denote the edge-boundary of $S$ in the 1-dimensional grid-graph $G_{k,1}$ on vertex-set $[k]$. We define
$$h^{*}: \mathcal{P}([k]) \to \mathbb{R};\quad h^{*}(S) = \begin{cases} 0 & \textrm{if } S = \emptyset \textrm{ or } S = [k];\\
\nu(S) \log_2 (|\partial S|/\nu(S)) & \textrm{otherwise.}\end{cases}$$
Note that $\nu(S)$ is simply the variance of $1_S$ with respect to the uniform measure on $[k]$.

\begin{theorem'}
Let $A \subset [k]^n$. Then for any $\epsilon >0$, $1_A$ is $\epsilon$-close to some
$M$-junta $g:[k]^n \to \{0,1\}$ , with
$$M \leq \exp \left( \frac{C_4}{\epsilon} \sum_{j=1}^n \mathbb{E}_{x \in [k]^{n-1}} h^{*}(A \cap a_j^x) \right),$$
where $C_4$ is an absolute constant. Here, $a_j^x$ denotes the fibre at $x$ in direction $j$, identified with $[k]$ in the natural way (so that $A \cap a_j^x$ is viewed as a subset of $[k]$).
\end{theorem'}

In this theorem, the bound on the junta-size depends upon both the edge-boundary and the measure of $A$ restricted to each fibre. We will see that this improves (up to constant factors) all the above-mentioned results (concerning the $\ell^1$-grid) simultaneously, so can be viewed a common generalization of them.

\subsection{Organization of the paper}

The rest of the paper is organized as follows. In Section~\ref{sec:junta}, we prove our main result, Theorem~\ref{thm:main}. In Section~\ref{sec:Lipschitz}, we prove the structure theorem for Lipschitz maps between discrete tori, Theorem~\ref{thm:lipschitz-tori}. In Section~\ref{sec:refinement}, we prove our refined Junta theorem. Finally, we conclude by mentioning some open problems in Section~\ref{sec:conclusion}.

\section{The junta theorem for the $\ell^1$-grid}
\label{sec:junta}
In this section, we will prove our main theorem.

\begin{theorem-main}
Let $A \subset [k]^n$. Then for any $\epsilon >0$, $1_A$ is $\epsilon$-close to some $M$-junta $g:[k]^n \to \{0,1\}$ , where
$$M \leq \exp(C_1|\partial A|/(k^{n-1}\epsilon))$$
for some absolute constant $C_1$. (In fact, one can take $C_1 = 16+8\sqrt{3 \ln 2}$.)
\end{theorem-main}

\begin{proof}
We first prove the theorem in the case $k=2^s$ for some $s \in \mathbb{N}$. Let $A \subset [k]^n$, and let $f = 1_A:\ [k]^n \to \{0,1\}$ denote its characteristic function. Let
$$\phi: \{0,1\}^s \to [k];\quad (x_1,\ldots,x_s) \mapsto 1+\sum_{i=1}^{s}x_i 2^{i-1}$$
denote the bijection corresponding to binary expansion. Let
$$\phi_{(n)}: (\{0,1\}^s)^n \to [k]^n;\quad ((x_{1,j},x_{2,j},\ldots,x_{s,j}))_{j \in [n]} \mapsto \left(1+\sum_{i=1}^{s} x_{i,j} 2^{i-1}\right)_{j \in [n]}$$
denote the bijection obtained by applying $\phi$ to each coordinate $j \in [n]$ separately. Then $\tilde{f}:= f \circ \phi_{(n)}:\ (\{0,1\}^{s})^{n} \to \{0,1\}$. We identify $(\{0,1\}^s)^{n}$ with $\{0,1\}^{sn}$ in the natural way, by partitioning $[sn]$ into $n$ blocks of size $s$,
$$B_j = \{(j-1)s+1,(j-1)s+2,\ldots,js\}\quad (j \in [n]);$$
the block $B_j$ corresponding to the coordinate $j$, for each $j \in [n]$.

Recall that for $j \in [n]$ and $x \in [k]^n$, we define the fibre at $x$ in direction $j$ by
$$a_j^x = \{y \in [k]^n:\ y_i=x_i\ \forall i \neq j\}.$$
We identify $a_j^x$ with $[k]$ in the natural way, via the bijection $y_j \mapsto y$. Let $F:[k] \to \{0,1\}$ denote the restriction of $f$ to $a_j^x$. For each $z \in [k]$, we let $z_i \in \{0,1\}$ denote the coefficient of $2^{i-1}$ in the binary expansion of $z$, i.e. $z_i = (\phi^{-1}(z))_i$.

We make the following claim.
\begin{claim}
\label{claim:binary}
Suppose $a_j^x$ contains exactly $m$ edges of $\partial A$. Then
$$\#\{(z,i) \in [k] \times [s]:\ z_i = 0,\ F(z) \neq F(z + 2^{i-1})\} \leq (k-1)m.$$
\end{claim}
\begin{proof}[Proof of Claim:]
Observe that for each pair $(z,i) \in [k] \times [s]$ with $F(z) \neq F(z + 2^{i-1})$, there exists $b \in [k-1]$ with $F(b) \neq F(b+1)$ and $z \leq b \leq z+2^{i-1}-1$. Let
$$\mathcal{S} = \{b\in [k-1]:\ F(b) \neq F(b+1)\}.$$
Since there are exactly $m$ edges of $\partial A$ within $a_j^x$, we have $|\mathcal{S}| = m$. Fix $b \in \mathcal{S}$; we will now bound
$$N(b) := \#\{(z,i) \in [k] \times [s]:\ b-2^{i-1}+1 \leq z \leq b\}.$$
We have
$$N(b) = \sum_{i=1}^{s} \#\{z \in [k]:\ b-2^{i-1}+1 \leq z \leq b\} \leq \sum_{i=1}^{s} 2^{i-1} = 2^s-1 = k-1,$$
for any $b \in [k-1]$. Therefore, we have
\begin{align*} \#\{(z,i) \in [k] \times [s]:\ F(z) \neq F(z + 2^{i-1})\} & \leq \sum_{b \in \mathcal{S}} N(b)\\
& \leq (k-1) |\mathcal{S}| 
 = (k-1)m,\end{align*}
proving the claim.
\end{proof}

Observe that
\begin{align*} \sum_{i=1}^{s} \Inf_i(F\circ \phi) & = 2^{-s-1} \#\{(z,i) \in [k] \times [s]:\ z_i = 0,\ F(z) \neq F(z + 2^{i-1})\}\\
& = (2/k) \#\{(z,i) \in [k] \times [s]:\ z_i = 0,\ F(z) \neq F(z + 2^{i-1})\}.\end{align*}
Hence, Claim \ref{claim:binary} implies that
$$\sum_{i=1}^{s} \Inf_i(F\circ \phi) \leq (2/k) (k-1)m \leq 2m.$$

By averaging this inequality over all $x \in [k]^n$, we obtain:
$$\sum_{i \in B_j} \Inf_i(\tilde{f}) \leq 2|\partial_j A|/k^{n-1},$$
where $\partial_j A$ denotes the set of edges of $\partial A$ in direction $j$. Summing this inequality over all $j \in [n]$ gives
$$\sum_{i=1}^{sn} \Inf_i(\tilde{f}) \leq 2|\partial A|/k^{n-1}.$$
Let $\epsilon >0$. By Theorem \ref{thm:friedgut-junta}, it follows that for every $\epsilon >0$, $\tilde{f}$ is $\epsilon$-close to some $M$-junta $\tilde{g}:\{0,1\}^{sn} \to \{0,1\}$, where
$$M \leq \exp(2C_0|\partial A|/(k^{n-1}\epsilon))$$
coordinates, where $C_0$ is the constant from Theorem \ref{thm:friedgut-junta}. Let $g = \tilde{g} \circ (\phi_{(n)})^{-1}$. Then $f$ is $\epsilon$-close to $g:[k]^n \to \{0,1\}$, and $g$ is also an $M$-junta. This proves the theorem in the case where $k$ is a power of 2.

We will now deduce the statement of the theorem for general $k \geq 2$. Suppose $k \geq 3$. Let $l>k$ be a power of $2$ such that
\begin{equation}\label{eq:size-cond} \left(1+\frac{k}{l-k}\right)^n \leq 2.\end{equation}
We consider a partition of the grid $[l]^n$ into $k^n$ large, rectangular `blocks' $(R_x:\ x \in [k]^n)$, where
$$R_x := \{y \in [l]^n:\ (\lceil y_1k/l\rceil, \lceil y_2k/l\rceil,\ldots,\lceil y_nk/l\rceil) = x\}.$$
Note that
$$\lfloor l/k \rfloor^n \leq |R_x| \leq \lceil l/k \rceil^n \quad \forall x \in [k]^n.$$
Let $A \subset [k]^n$. Define
$$\breve{A} = \cup_{x \in A} R_x \subset [l]^n.$$
In other words, the subset $\breve{A} \subset [l]^n$ is the union of all the blocks corresponding to elements of $A$. Observe that each edge of $\partial \breve{A}$ arises from exactly one edge of $\partial A$, and for each edge of $\partial A$, there are between $\lfloor l/k \rfloor^{n-1}$ and $\lceil l/k \rceil^{n-1}$ edges of $\partial \breve{A}$ arising from it. Hence,
$$|\partial \breve{A}| \leq \lceil l/k \rceil^{n-1} |\partial A| < (l/k+1)^{n-1} |\partial A|,$$
and therefore
$$\frac{|\partial \breve{A}|}{l^{n-1}} < (1+k/l)^{n-1} \frac{|\partial A|}{k^{n-1}} \leq 2 \frac{|\partial A|}{k^{n-1}},$$
using (\ref{eq:size-cond}). Applying our result above for grids of side-length a power of 2, we see that $1_{\breve{A}}$ is $(\epsilon/2)$-close to some $M$-junta $\breve{g}:[l]^n \to \{0,1\}$, where
$$M \leq \exp(4C_0|\partial \breve{A}|/(l^{n-1}\epsilon)) \leq \exp(8C_0|\partial A|/(k^{n-1}\epsilon)).$$
Since the function $1_{\breve{A}}$ is constant on each block, we may assume without loss of generality that the $M$-junta $\breve{g}$ is also constant on each block. Hence, $\breve{g}$ defines an $M$-junta $g:[k]^n \to \{0,1\}$, where $g(x)=\breve{g}(R_x)$, the value of $\breve{g}$ on the block $R_x$. Let $C \subset [k]^n$ and $\breve{C} \subset [l]^n$ be the subsets with characteristic functions $g$ and $\breve{g}$ respectively. Since $|R_x| \geq \lfloor l/k \rfloor^n$ for all $x \in [k]^n$, we have
$$|\breve{A} \Delta \breve{C}| \geq \lfloor l/k \rfloor^n |A \Delta C| \geq (l/k-1)^n |A \Delta C|,$$
and therefore
$$||1_{A}-g||_1 = \frac{|A \Delta C|}{k^n} \leq \left(1+\frac{k}{l-k}\right)^n \frac{|\breve{A} \Delta \breve{C}|}{l^n} = \left(1+\frac{k}{l-k}\right)^n\epsilon/2 \leq \epsilon,$$
using (\ref{eq:size-cond}). Hence, $1_{A}$ is $\epsilon$-close to the $M$-junta $g: [k]^n \to \{0,1\}$. This completes the proof of the theorem.
\end{proof}

We now observe that Theorem \ref{thm:main} is sharp up to the value of the absolute constant in the exponent. To see this, we use a simple modification of the construction in \cite{friedgut-junta}, which in turn was based on the `tribes' construction of Ben-Or and Linial \cite{bol}.

\begin{example}
\label{example:tribes}
Let $k$ be even, let $L \geq 2$, let $s = \log_2(1/(6\epsilon))$, let $t = (L-1)e/(6\epsilon)$, and let $n = 1+s+t2^t$. (To avoid the need for floor and ceiling signs, let us assume that $s$ and $t$ are integers; clearly, we may choose $\epsilon >0$ arbitrarily small and $L \geq 2$ arbitrarily large, such that this holds.) Partition the $t2^t$ coordinates
$$\{s+2,\ldots,s+t2^t+1\}$$
into $2^t$ `tribes' $T_1,\ldots,T_{2^t}$ of size $t$ each. Now define
\begin{align*}
A & = \{x \in [k]^n:\ x_1 \leq k/2\}\\
& \cap \{x \in [k]^n:\ x_j \leq k/2\ \textrm{for some } j \in \{2,\ldots,s+1\},\\
& \textrm{or }\exists i \in [2^t]:\ x_j \leq k/2\ \forall j \in T_i\}.\end{align*}
Then we have $\tfrac{1}{2}-3\epsilon < |A|/k^n < \tfrac{1}{2}$, and $|\partial A| \leq (1+O(\epsilon \log_2(1/\epsilon))) Lk^{n-1}$, but if $g$ is an $M$-junta with $||1_A - g||_2^2 \leq \epsilon$, then
$$M \geq (1-e/3-o(1))t2^t,$$
where $o(1)$ denotes a function tending to $0$ as $t \to \infty$. Hence, $M \geq 2^{L/(6 \epsilon)}$ for all sufficiently large $L$.
\end{example}

\section{Lipschitz maps between discrete tori}
\label{sec:Lipschitz}

In this section, we show how to use Theorem \ref{thm:main} to prove Theorem \ref{thm:lipschitz-tori}, our structure theorem for Lipschitz maps between two discrete tori.

We say a function $f: \mathbb{Z}_k^{n} \to \mathbb{Z}_l^{m}$ is $\alpha$-{\em Lipschitz (with respect to the $L^1$-norm)} if for any $x,y \in \mathbb{Z}_k^n$, we have $||f(x)-f(y)||_1 \leq \alpha ||x-y||_1$, i.e.
\begin{equation}\label{eq:lipcond} \frac{1}{m} \sum_{i=1}^{m} |f_i(x)-f_i(y)|' \leq \alpha \frac{1}{n} \sum_{i=1}^{n} |x_i-y_i|'\quad \forall x,y \in \mathbb{Z}_n^k.\end{equation}
(Here, as before, $|\cdot|'$ denotes cyclic distance.) In particular, for all $x \in \mathbb{Z}_k^n$ and all $j \in [n]$, we have
$$\frac{1}{m} \sum_{i=1}^{m} |f_i(x)-f_i(x\oplus e_j)|' \leq \alpha/n,$$
where $e_j$ denotes the $j$th unit vector in $\mathbb{Z}_k^n$, and $\oplus$ denotes modulo-$k$ addition. (It is easy to see that this property is in fact equivalent to (\ref{eq:lipcond}).) Taking expectations with respect to a uniform random $x \in \mathbb{Z}_k^n$ and a uniform random $j \in [n]$, yields the following.
$$\frac{1}{m} \sum_{i=1}^{m} \frac{1}{n} \sum_{j=1}^{n} \mathbb{E}_x [|f_i(x)-f_i(x\oplus e_j)|'] \leq \alpha/n.$$
Cancelling out a factor of $1/n$ on each side, we get
$$\frac{1}{m} \sum_{i=1}^{m}\left( \sum_{j=1}^{n} \mathbb{E}_x [|f_i(x)-f_i(x\oplus e_j)|'] \right)\leq \alpha.$$

Therefore, by Markov's inequality, for any $\delta >0$, we have
\begin{equation}\label{eq:sum-inf} \sum_{j=1}^{n} \mathbb{E}_x [|f_i(x)-f_i(x\oplus e_j)|'] \leq \alpha/\delta\end{equation}
for all but at most $\delta m$ coordinates $i \in [m]$.

We note that the above trick of taking the expectation of the Lipschitz condition was first used in \cite{austin}, and independently in \cite{b-s}.

We now obtain a result describing the structure of functions $f: \mathbb{Z}_k^n \to \mathbb{Z}_l$ with
$$\sum_{j=1}^{n} \mathbb{E}_x [|f(x)-f(x\oplus e_j)|'] \leq B.$$

\begin{theorem}
\label{thm:sum-inf}
Suppose $f:\mathbb{Z}_k^n \to \mathbb{Z}_l$ with
$$\sum_{j=1}^{n} \mathbb{E}_x [|f(x)-f(x\oplus e_j)|'] \leq B.$$
Then for any $\epsilon >0$, there exists an $M$-junta $g:\mathbb{Z}_k^n \to \mathbb{Z}_l$ such that $f$ is $\epsilon$-close to $g$, and
 $$M \leq l\exp(C_2 B /\epsilon),$$
where $C_2$ is an absolute constant. (In fact, one can take $C_2 = 4C_1$, where $C_1$ is the constant from Theorem \ref{thm:main}.)
\end{theorem}
\begin{proof}
If $h:\mathbb{Z}_k^n \to \mathbb{Z}_l$, then for each $t \in \mathbb{Z}_l$, define a Boolean function $h^{(t)}:\mathbb{Z}_k^n \to \{0,1\}$ by
$$h^{(t)}(x) = 1\{h(x)\in \{t,t+1,\ldots,t+\lfloor l/2 \rfloor -1\}\},$$
where addition is modulo $l$. Then we have
$$|h(x) - h(y)|' = \tfrac{1}{2} \sum_{t \in \mathbb{Z}_l} |h^{(t)}(x) - h^{(t)}(y)|\quad \forall x,y \in \mathbb{Z}_k^n,$$
and for any two functions $h,\tilde{h}:\mathbb{Z}_k^n \to \mathbb{Z}_l$, we have
\begin{align}
\label{eq:norm-relation}
||h-\tilde{h}||_1 & = \frac{1}{k^n}\sum_{x \in \mathbb{Z}_k^n} |h(x)-\tilde{h}(x)|' \nonumber \\
& =  \frac{1}{k^n} \sum_{x \in \mathbb{Z}_k^n} \tfrac{1}{2} \sum_{t \in \mathbb{Z}_l} |h^{(t)}(x) - \tilde{h}^{(t)}(x)|\nonumber \\
& = \tfrac{1}{2} \sum_{t \in \mathbb{Z}_l} \frac{1}{k^n} \sum_{x \in \mathbb{Z}_k^n} |h^{(t)}(x) - \tilde{h}^{(t)}(x)| \nonumber \\
& =  \tfrac{1}{2} \sum_{t \in \mathbb{Z}_l} ||h^{(t)} - \tilde{h}^{(t)}||_1.\end{align}
Let $f$ be as in the statement of the theorem. Then
\begin{align}
\label{eq:sum-bound}
B & \geq \sum_{j=1}^n \E_x|f(x) - f(x\oplus e_j)|' \nonumber \\
& = \tfrac{1}{2} \sum_{t \in \mathbb{Z}_l} \sum_{j=1}^n \E_x |f^{(t)}(x) -f^{(t)}(x\oplus e_j)| \nonumber \\
& \geq \tfrac{1}{4} \sum_{t \in \mathbb{Z}_l} \frac{|\partial' A^{(t)}|}{k^{n}},\end{align}
where
$$A^{(t)} = \{x \in \mathbb{Z}_k^n:\ f(x) \in \{t,t+1,\ldots,t+\lfloor l/2 \rfloor -1\}\} = \{x \in \mathbb{Z}_k^n:\ f^{(t)}(x)=1\}.$$
Let $\epsilon_1,\epsilon_2,\ldots, \epsilon_{l} >0$ to be chosen later. By Corollary \ref{corr:torus-junta}, for each $t \in \mathbb{Z}_l$, there exists a junta $g_t:\mathbb{Z}_k^n \to \{0,1\}$, such that $||f^{(t)}-g_t||_1 \leq \epsilon_t$ and $g_t$ depends upon a set ($J_t$, say) of at most $\exp(C_1 |\partial' A^{(t)}|/(k^{n-1}\epsilon_t))$ coordinates. Let
$$J = \cup_{t \in \mathbb{Z}_l} J_t.$$
Choose a $J$-junta $h:\mathbb{Z}_k^n \to \mathbb{Z}_l$ such that
$$\sum_{t \in \mathbb{Z}_l} ||h^{(t)}-g_t||_1$$
is minimal. An easy averaging argument now implies that
\begin{equation} \label{eq:small-difference} \sum_{t \in \mathbb{Z}_l}||h^{(t)} - g_t||_1 \leq \sum_{t \in \mathbb{Z}_l}||f^{(t)} - g_t||_1.\end{equation}
Indeed, writing $\sum_{t \in \mathbb{Z}_l}||f^{(t)} - g_t||_1 =: \epsilon'$, we have
\begin{align*} \epsilon' & = \sum_{t \in \mathbb{Z}_l}||f^{(t)} - g_t||_1\\
& = \sum_{t \in \mathbb{Z}_l} \mathbb{E}_{y \in \mathbb{Z}_k^{[n] \setminus J}} \mathbb{E}_{z \in \mathbb{Z}_k^J} |f^{(t)}(y,z) - g_t(y,z)|\\
& = \mathbb{E}_{y \in \mathbb{Z}_k^{[n] \setminus J}} \sum_{t \in \mathbb{Z}_l} \mathbb{E}_{z \in \mathbb{Z}_k^J} |f^{(t)}(y,z) - g_t(y,z)|,
\end{align*}
so by averaging, there exists $y_0 \in \mathbb{Z}_k^{[n] \setminus J}$ such that
$$\sum_{t \in \mathbb{Z}_l} \mathbb{E}_{z \in \mathbb{Z}_k^J} |f^{(t)}(y_0,z) - g_t(y_0,z)| \leq \epsilon';$$
defining $\tilde{f}(y,z) = f(y_0,z)$ for all $y \in \mathbb{Z}_k^{[n] \setminus J}$ and all $z \in \mathbb{Z}_k^{J}$, we obtain
$$\sum_{t \in \mathbb{Z}_l} ||\tilde{f}^{(t)} - g_t ||_1 = \mathbb{E}_{z \in \mathbb{Z}_k^J} |f^{(t)}(y_0,z) - g_t(y_0,z)| \leq \epsilon'.$$
Since $\tilde{f}$ is a $J$-junta, by the minimality property of $h$, we must have
$$\sum_{t \in \mathbb{Z}_l}||h^{(t)} - g_t||_1 \leq \epsilon',$$
proving (\ref{eq:small-difference}).

Hence, using (\ref{eq:norm-relation}), we have
\begin{align*} ||f-h||_1 & = \tfrac{1}{2} \sum_{t \in \mathbb{Z}_l}||f^{(t)} - h^{(t)}||_1\\
& \leq \tfrac{1}{2} \sum_{t \in \mathbb{Z}_l}||f^{(t)} - g_t||_1 +\tfrac{1}{2} \sum_{t \in \mathbb{Z}_l}||h^{(t)} - g_t||_1\\
& \leq \sum_{t \in \mathbb{Z}_l}||f^{(t)} - g_t||_1\\
& \leq \sum_{t \in \mathbb{Z}_l} \epsilon_t.\end{align*}
Now choose
$$\epsilon_t = \frac{\epsilon |\partial' A^{(t)}|}{4 k^{n} B}.$$
Then by (\ref{eq:sum-bound}),
$$\sum_{t \in \mathbb{Z}_l} \epsilon_t \leq \epsilon,$$
and
$$|J| \leq l \exp(4C_1 B k/\epsilon).$$
Hence, $f$ is $\epsilon$-close to the $J$-junta $h$, proving the theorem.
\end{proof}

Theorem \ref{thm:lipschitz-tori} follows immediately from (\ref{eq:sum-inf}) and Theorem \ref{thm:sum-inf}.

To prove Theorem \ref{thm:lipschitz-grid}, one first averages the Lipschitz condition over all edges of the grid, giving
\begin{equation}\label{eq:averaging} \frac{1}{m} \sum_{i=1}^{m}\left( \sum_{j=1}^{n} \mathbb{E}_{\textrm{edges } \{x,y\}\atop \textrm{in direction }j}|f_i(x)-f_i(y)|\right)\leq \alpha.\end{equation}
One then proves the following analogue of Theorem \ref{thm:sum-inf} for the grid.
\begin{theorem}
\label{thm:sum-inf-grid}
Suppose $f:[k]^n \to [l]$ with
$$\sum_{j=1}^{n} \mathbb{E}_{\textrm{edges } \{x,y\}\atop \textrm{in direction }j} [|f(x)-f(y)|] \leq B.$$
Then for any $\epsilon >0$, there exists an $M$-junta $g:\mathbb{Z}_k^n \to \mathbb{Z}_l$ such that $f$ is $\epsilon$-close to $g$, and
 $$M \leq l\exp(C_3 B k /\epsilon),$$
 where $C_3$ is an absolute constant. (In fact, one can take $C_3 = 2C_1$, where $C_1$ is the constant from Theorem \ref{thm:main}.)
\end{theorem}
\begin{proof}[Proof of Theorem \ref{thm:sum-inf-grid}]
If $h:[k]^n \to[l]$, then for each $t \in [l-1]$, define a Boolean function $h^{(>t)}:[k]^n \to \{0,1\}$ by
$$h^{(> t)}(x) = 1\{h(x) > t\}.$$
Then we have
$$h = 1+\sum_{t=1}^{l-1} h^{(>t)},$$
and
$$|h(x) - h(y)| = \sum_{t \in [l-1]} |h^{(>t)}(x) - h^{(>t)}(y)|\quad \forall x,y \in [k]^n.$$
Also, for any two functions $h,\tilde{h}:[k]^n \to [l]$, we have
\begin{align*}
||h-\tilde{h}||_1 & = \frac{1}{k^n}\sum_{x \in [k]^n} |h(x)-\tilde{h}(x)| \\
& =  \frac{1}{k^n} \sum_{x \in [k]^n} \sum_{t=1}^{l-1} |h^{(>t)}(x) - \tilde{h}^{(>t)}(x)| \\
& = \sum_{t =1}^{l-1} \frac{1}{k^n} \sum_{x \in [k]^n} |h^{(>t)}(x) - \tilde{h}^{(>t)}(x)| \\
& =  \sum_{t =1}^{l-1} ||h^{(>t)} - \tilde{h}^{(>t)}||_1.\end{align*}
Let $f$ be as in the statement of the theorem. Then
\begin{align}
\label{eq:sum-bound-grid}
B & \geq \sum_{j=1}^{n} \mathbb{E}_{\textrm{edges } \{x,y\}\atop \textrm{in direction }j} [|f(x)-f(y)|] \nonumber \\
& = \sum_{t=1}^{l-1} \sum_{j=1}^n \E_{\textrm{edges } \{x,y\}\atop \textrm{in direction }j} |f^{(>t)}(x) -f^{(>t)}(y)| \nonumber \\
& = \sum_{t=1}^{l-1} \frac{|\partial A^{(>t)}|}{k^{n}-k^{n-1}}\nonumber \\
& \geq \tfrac{1}{2} \sum_{t=1}^{l-1} \frac{|\partial A^{(>t)}|}{k^{n}}\end{align}
where
$$A^{(>t)} = \{x \in [k]^n:\ f(x) >t\} = \{x \in [k]^n:\ f^{(>t)}(x)=1\},$$
i.e., the $A^{(>t)}$ are the `level sets' of the function $f$. Let $\epsilon_1,\epsilon_2,\ldots, \epsilon_{l} >0$ to be chosen later. By Theorem \ref{thm:main}, for each $t \in [l-1]$, there exists a junta $g_t:[k]^n \to \{0,1\}$, such that $||f^{(>t)}-g_t||_1 \leq \epsilon_t$ and $g_t$ depends upon a set ($J_t$, say) of at most $\exp(C_1 |\partial' A^{(t)}|/\epsilon_t)$ coordinates. Let
$$J = \cup_{t \in \mathbb{Z}_l} J_t.$$
Define
$$g = 1+\sum_{t=1}^{l-1} g_t.$$
Then $g$ is a $J$-junta, and
$$||f-g||_1 = \sum_{t \in [l-1]} ||f^{(>t)} - g_t|| \leq \sum_{t=1}^{l-1}\epsilon_t.$$
Now choose
$$\epsilon_t = \frac{\epsilon |\partial A^{(>t)}|}{2 k^{n} B}.$$
Then by (\ref{eq:sum-bound-grid}),
$$\sum_{t=1}^{l-1} \epsilon_t \leq \epsilon,$$
and
$$|J| \leq (l-1) \exp(2 C_1 B/\epsilon).$$
Hence, $f$ is $\epsilon$-close to the $J$-junta $g$, proving Theorem \ref{thm:sum-inf-grid}.
\end{proof}
Theorem \ref{thm:lipschitz-grid} follows immediately from (\ref{eq:averaging}) and Theorem \ref{thm:sum-inf-grid}.

\section{A refined Junta theorem for the $\ell^1$-grid}
\label{sec:refinement}

\subsection{Previous results}

Before discussing our refined junta theorem, we first describe some previously obtained junta theorems for the $\ell^1$-grid. All of these results, as well as our result, can be stated in terms of the restriction of the set to fibres. If $f:[k]^n \to \{0,1\}$, let $f^x_j:[k] \to \{0,1\}$ denote the restriction of $f$ to $a_j^x$, which is naturally identified with $[k]$. (Note that in the proof of Theorem \ref{thm:main}, where this causes no ambiguity, this restriction was simply denoted by $F$.) The prior theorems can be summarised as follows. (Note that the general structure of each result is the same; they differ only in the definition of the function $h$ used inside the expectation. This viewpoint is a discrete analogue of the notion of `$h$-influence' defined in~\cite{keller}.)
\begin{theorem}
\label{thm:prior}
Let $A \subset [k]^n$, and let $f=1_{A}$. Then for any $\epsilon >0$, $1_A$ is $\epsilon$-close to some $M$-junta $g:[k]^n \to \{0,1\}$ , with
$$M \leq \exp \left( \frac{C}{\epsilon} \sum_{j=1}^n \mathbb{E}_{x \in [k]^{n-1}} h(f_j^x) \right),$$ where $C$ is an absolute constant, and $h:\{0,1\}^{[k]} \to \mathbb{R}$ is defined as follows:
\begin{enumerate}
\item (Dinur and Friedgut, Theorem~2.12 of~\cite{fd}, only in the monotone case)
\[
h(F) = \left\lbrace
  \begin{array}{l l}
    0 & \mbox{ if }F \mbox{ is constant;}\\
    1 & \mbox{otherwise.}
  \end{array}
\right.
\]

\item (Dinur and Friedgut, Section~2.2 of~\cite{fd}, in the general case; this is a very easy consequence of the original Junta theorem)
\[
h(F) = \left\lbrace
  \begin{array}{l l}
    0 & \mbox{if }F \mbox{ is constant;}\\
    \log_2 k & \mbox{otherwise.}
  \end{array}
\right.
\]

\item (Hatami, Lemma~3.5 of~\cite{hatami})
\[
h(F)= (\log_2 k) \cdot \Var(F) = (\log_2 k)\cdot \Pr[F(y)=1](1-\Pr[F(y)=1]),\]
 where the probability is taken w.r.t. the uniform measure on $[k]$.

\item (Keller, Proposition~3.8 of~\cite{keller}, only in the monotone case)
\[h(F)= \mathrm{Ent}(\Pr[F(y)=1]),\] where
$\mathrm{Ent}(p)= -p \log_2 p - (1-p) \log_2(1-p)$
is the binary entropy function.

\item (Sachdeva and Tulsiani, Theorem~3.6 of~\cite{s-t})
\[h(F) = k |\partial F|,\] where $\partial F$ denotes the edge-boundary of the subset $\{y:F(y)=1\} \subset [k]$ in the 1-dimensional grid $G_{k,1}$, so that
$|\partial F| = |\{y \in [k-1]:\ F(y+1) \neq F(y)\}|$.

\item (Theorem~\ref{thm:main} above)
\[
h(F) =|\partial F|.\]
\end{enumerate}
(Note that the values of the absolute constant $C$ implicit in these theorems, differ from one another.)
\end{theorem}
It should be mentioned that Dinur and Friedgut conjectured in \cite{fd} that
assertion (1) above holds also for arbitrary subsets of $[k]^n$. However, Hatami ~\cite{hatami} disproved this conjecture, constructing an example (based on a decision tree) which shows that if $h$ in the theorem is of the form
$$h(F) = \gamma(k) 1\{F \mbox{ is non-constant}\},$$
then we must have $\gamma(k) \geq c \log_2 k$. (This also follows from Example \ref{example:decision-tree}, below.) Hence, assertion (2) above is sharp up to the value of the absolute constant $C$.

Note also that while several of the statements are strictly stronger than others (for example, (4) is strictly
stronger than (1)), there are several pairs of statements which are incomparable. In particular, our
main theorem (6) does not necessarily improve over bound (2), since $|\partial (f_j^x)|$
can be as large as $k-1$ (for example, if $f(x) = \parity(\sum_{i=1}^{n}x_i)$), whereas the function $h$ in (2) is at most $\log_2 k$.

\subsection{The refined junta theorem}

In this section, we prove our refined junta theorem for the $\ell^1$-grid. (For ease of comparison with the prior results, we state it this time in the same form as Theorem \ref{thm:prior}.)
\begin{theorem}
\label{thm:refined-junta}
Let $A \subset [k]^n$, and let $f = 1_{A}$. Then for any $\epsilon >0$, $f$ is $\epsilon$-close to some
$M$-junta $g:[k]^n \to \{0,1\}$ , with
$$M \leq \exp \left( \frac{C_4}{\epsilon} \sum_{j=1}^n \mathbb{E}_{x \in [k]^{n-1}} h^{*}(f_j^x) \right),$$
where $C_4$ is an absolute constant, and
$$h^*: \{0,1\}^{[k]} \to \mathbb{R};\quad h^*(F) = \begin{cases} 0 & \mbox{if }F \mbox{ is constant};\\
\Var(F) \log_2{ \frac{|\partial F|}{\Var(F)}} & \mbox{otherwise.}\end{cases}$$
\end{theorem}

The following easy claim shows that Theorem~\ref{thm:refined-junta} generalizes all six bounds given above (up to an absolute constant factor).
\begin{claim}
There exists an absolute constant $c>0$, such that for all functions $F:[k] \to \{0,1\}$,
$h^*(F) \leq c \cdot h(F)$, for each of the six definitions of $h$ given in Theorem \ref{thm:prior}.
\end{claim}

\begin{proof}
If $F:[k] \to \{0,1\}$ is non-constant, then
\begin{align}
\label{eq:F-non-constant}
h^*(F) &= \Var(F) \log_2{ \frac{|\partial F|}{\Var(F)}} \nonumber\\
& = \Var(F) \left(\log_2 |\partial F| +
\log_2{ \frac{1}{\Var(F)}}\right) \nonumber \\
&\leq 3\Var(F) \log_2 k,
\end{align}
where the last inequality holds since $\log_2 |\partial F| \leq \log_2 k$ and $\log_2(1/\Var(F)) \leq \log_2(k^2/(k-1)) \leq 2 \log_2 k$. Hence, if $f_j^x$ is constant, then $h^{*}(f_j^x)=0$, and if not,
then $h^{*}(f_j^x) \leq c\Var(f_j^x) \cdot \log_2 k$. This shows that
Theorem~\ref{thm:refined-junta} generalizes assertions (2) and (3) above.

In the monotone case, if $f_j^x$ is non-constant then we have $|\partial (f_j^x)|=1$, and therefore
\[h^{*}(f_j^x) = \Var(f_j^x) \log_2{\frac{1}{\Var(f_j^x)}} < \mathrm{Ent}(\Pr[f_j^x(y)=1]) < 1,\]
which shows that Theorem~\ref{thm:refined-junta} generalizes assertions (1) and (4) above.

Finally, Theorem~\ref{thm:refined-junta} generalizes assertions (5) and (6) above, since if $F:[k] \to \{0,1\}$ is non-constant, then
\[
h^{*}(F) = \Var(F) \left(\log_2 |\partial F| + \log_2{ \frac{1}{\Var(F)}}\right) \leq
\tfrac{1}{4}\log_2 |\partial F| +\tfrac{1}{2} \leq |\partial F|,
\]
using the facts that $\Var(F) \leq 1/4$, that the function $t \mapsto t \log_2 (1/t)$ is increasing on $(0,1/4)$, and that $|\partial F| \geq 1$.
\end{proof}

It should be noted that for each of the statements 2, 3, 4, 5 and 6 above, there are cases in which
the assertion of Theorem~\ref{thm:refined-junta} outperforms them significantly. For example, if on each fibre where $f_j^x$ is non-constant, the size of the restricted boundary is constant but the restricted measure is far from 0 and 1, then
Theorem~\ref{thm:refined-junta} is stronger than statement 2 and statement 3 by a factor of $\Omega(\log_2 k)$ in the exponent. This is the case in Example \ref{example:tribes}, for which Theorem \ref{thm:main} (and therefore Theorem \ref{thm:refined-junta}) is sharp up to a constant factor in the exponent.

On the other hand, if on each fibre where $f_j^x$ is non-constant, the restricted boundary has size $\Theta(k)$, then Theorem~\ref{thm:refined-junta}
is stronger than Theorem \ref{thm:main} by a factor of $\Omega(k/(\log_2 k))$ in the exponent. To see that this may occur, we use a `random' variant of Hatami's decision tree construction in \cite{hatami}, which we now describe in detail.

\begin{example}
\label{example:decision-tree}
We first recall the notion of a function defined by a decision tree. Let $T$ be a rooted, finite, directed $k$-ary tree (i.e., each non-leaf node has $k$ children). Suppose that each non-leaf node $v$ is labelled with an index $i_v \in [n]$, and each of the $k$ directed edges emanating from $v$ is labelled with a different element of $\{1,2,\ldots,k\}$. Suppose further that each leaf $w$ is labelled with a value $\val(w) \in \{0,1\}$. This (labelled) decision tree $T$ defines a function $f_{T}:[k]^n \to \{0,1\}$ as follows. Fix $x \in [k]^n$. To specify $f(x)$, we define a directed path $P = P(x)$ in $T$ as follows. Start at the root, and whenever we are at a non-leaf vertex $v$, look at the index $i_v$ and go to the child $v'$ of $v$ such that the directed edge $(v,v')$ has label $x_{i_v}$. Continue until we reach a leaf $w$; then let $f(x) = \val(w)$.

Now we show that there exists a decision tree with the desired property. Let $d \geq 2$, and let $k$ be even. Choose $T$ to be a $k$-ary tree of depth $d$, meaning that each leaf lies on a path of length $d$ from the root, so that there are $k^d$ leaves. Choose $n = \sum_{s = 0}^{d-1}k^s$ to be the number of non-leaf vertices of $T$, and label each non-leaf vertex $v$ of $T$ with a different index $i_v \in [n]$ (using an arbitrary bijection). Let $\mathcal{L}(T)$ be the set of leaves of $T$. We now choose a function $\val: \mathcal{L}(T) \to \{0,1\}$ at random, as follows. For each leaf-parent $v$, let $C(v)$ denote its set of $k$ children (these are all leaves). For each leaf-parent $v$ independently, choose a $(k/2)$-subset $S(v) \subset C(v)$ (uniformly at random from all ${k \choose k/2}$ such sets), and define
$$\val(w) = \begin{cases} 1 & \textrm{ if } w \in S(v);\\
0 & \textrm{ if } w \in C(v) \setminus S(v).\end{cases}$$
We let $\mathcal{V}$ denote the set of all valuations obtained in this way; note that the function $\val$ is chosen uniformly at random from the set $\mathcal{V}$. (This may be compared with Hatami's deterministic example in \cite{hatami}, where $\val(w) = \lfloor 2a/k \rfloor$, if $a$ is the label of the edge between $w$ and its parent node.)

We now show that the associated function $f_T$ has the desired property with positive probability. Fix an index $j \in [n]$ and suppose it labels a node at distance $s$ from the root. (Abusing notation slightly, if $x \in [k]^{n-1}$ and $x_j \in [k]$, we will write $(x,x_j)$ for the vector produced from $x$ by inserting $x_j$ between the $(j-1)$th and the $j$th entries of $x$.)

Observe that for each $x \in [k]^{n-1}$ and each $x_j \in [k-1]$, we have $f_T(x,x_j) \neq f_T(x,x_j+1)$ if and only if (I) the path $P = P(x,x_j)$ goes through the node labelled with $j$, {\em and} (II) the leaves $w$ and $w'$ at which the paths $P(x,x_j)$ and $P(x,x_j+1)$ (respectively) end, have $\val(w) \neq \val(w')$.

First, assume that $s \leq d-2$, i.e. that $j$ does not label a leaf-parent. The event (I) depends only on $x$ (not on $x_j$), and has probability $1/k^s$, since $P$ is equally likely to go through any of the $k^s$ nodes at distance $s$ from the root. Conditional on the event (I) occurring, the event (II) occurs with probability $1/2$, since if (I) occurs, then the paths $P(x,x_j)$ and $P(x,x_j+1)$ go through different leaf-parents, so $\val(w)$ and $\val(w')$ are independent Bernoulli($\tfrac{1}{2}$) random variables, and so $\Pr_{\val}\{\val(w) \neq \val(w')\} = 1/2$. Hence,
$$\Pr_{\val} \Pr_{x \in [k]^{n-1}} \{f_T(x,x_j) \neq f_T(x,x_j+1)\} = \frac{1}{k^s} \cdot \frac{1}{2} = \frac{1}{2k^s}.$$
Summing over all $x_j \in [k-1]$, the linearity of expectation now yields
$$\mathbb{E}_{\val} \mathbb{E}_{x \in [k]^{n-1}} |\partial((f_T)_{j}^x)| = (k-1) \cdot \frac{1}{2k^s} = \frac{k-1}{2k^s}$$
for all $j \in [n]$ labelling a node at distance $s \leq d-2$ from the root.

Now assume that $s=d-1$, so that $j$ labels a leaf-parent, $v$ say. Then as before, the event (I) depends only on $x$ (not on $x_j$), and has probability $1/k^{d-1}$, since $P$ is equally likely to go through any of the $k^{d-1}$ nodes at distance $d-1$ from the root. Conditional on the event (I) occurring, the leaves $w$ and $w'$ are distinct children of $v$, so the probability of the event (II) occurring is now
\begin{align*} \Pr\{\val(w) \neq \val(w')\} & = 1-\Pr\{w,w' \in S(v)\} - \Pr\{w,w' \in C(v) \setminus S(v)\}\\
&= 1-2\cdot \frac{{k-2 \choose k/2-2}}{{k \choose k/2}}\\
& =\frac{k}{2(k-1)}.\end{align*}
Summing over all $x_j \in [k-1]$, the linearity of expectation yields
$$\mathbb{E}_{\val} \mathbb{E}_{x \in [k]^{n-1}} |\partial((f_T)_{j}^x)| = (k-1) \cdot \frac{1}{k^{d-1}} \cdot \frac{k}{2(k-1)} = \frac{1}{2k^{d-2}}$$
for all $j \in [n]$ labelling a node at distance $d-1$ from the root.

There are $k^s$ indices $j \in [n]$ of distance $s$ from the root, for each $s \in \{0,1,2,\ldots,d-1\}$. Summing over all indices $j \in [n]$ therefore gives
\begin{align*} \mathbb{E}_{\val} \left[\sum_{j=1}^{n} \mathbb{E}_{x \in [k]^{n-1}} |\partial((f_T)_{j}^x)| \right]& = \sum_{j=1}^{n} \mathbb{E}_{\val} \mathbb{E}_{x \in [k]^{n-1}} |\partial((f_T)_{j}^x)|\\
& = \sum_{s=0}^{d-2} k^s \cdot \frac{k-1}{2k^s} + k^{d-1} \cdot \frac{1}{2k^{d-2}} \\
&= \tfrac{1}{2}(d-1)(k-1) + \tfrac{1}{2} k.\end{align*}
Hence, there exists a valuation $\val'$ such that
$$\sum_{j=1}^{n} \mathbb{E}_{x \in [k]^{n-1}} |\partial((f_T)_{j}^x)| \geq \tfrac{1}{2}(d-1)(k-1)+ \tfrac{1}{2}k \geq \tfrac{1}{4}dk,$$
when $f_{T}$ is defined using $\val'$.

On the other hand, for each fibre $a_j^x$ on which $f_T$ is non-constant, we have
$$h^*((f_T)_{j}^x) \leq 3 \log_2(k),$$
by (\ref{eq:F-non-constant}). For an arbitrary valuation $\val$, $f_T$ is non-constant on $a_j^x$ only if the path $P = P(x,x_j)$ goes through the node labelled with $j$; this happens with probability $1/k^s$, where $s$ is the distance from the root to the node labelled with $j$, by exactly the same argument as above. Therefore, we have
$$\mathbb{E}_{x \in [k]^{n-1}} h^*((f_T)_{j}^x) \leq \frac{3\log_2(k)}{k^s}$$
for all $j \in [n]$ labelling a node at distance $s$ from the root, so
$$\sum_{j=1}^{n} \mathbb{E}_{x \in [k]^{n-1}} h^*((f_T)_{j}^x) \leq \sum_{s=0}^{d-1} k^s \cdot \frac{3\log_2 (k)}{k^s} = 3d\log_2(k).$$
This holds for an arbitrary valuation $\val$, so in particular for $\val'$.

We may conclude that there exists a function $f_T$ (namely, the decision-tree function defined by $\val'$) for which the bounds on the sizes of the approximating juntas in Theorems \ref{thm:main} and \ref{thm:refined-junta} differ by a factor of $\Omega(k/\log_2 k)$, as desired.

It is also easy to see that if $g:[k]^n \to \{0,1\}$ is an $M$-junta with $||f_T - g||_1 < 1/4$, then $M > k^{d-1}/2$. (Exactly the same argument as in \cite{hatami} can be used. Namely, if $g:[k]^n \to \{0,1\}$ is an $M$-junta, where $M \leq k^{d-1}/2$, then $g$ must remain constant whenever $x_{i_v}$ is varied and all other coordinates are fixed, for at least half of all leaf-parents $v$. On the other hand, for any valuation $\val \in \mathcal{V}$, if $v$ is any leaf-parent, and $x_{i_v}$ is varied while all the other coordinates are fixed, then $f_T$ takes the value $0$ exactly $k/2$ times, and the value 1 exactly $k/2$ times.) Hence, for the function $f_T$, Theorem \ref{thm:refined-junta} is sharp up to a constant factor in the exponent.
\end{example}

\subsection{Proof of the refined junta theorem (Theorem \ref{thm:refined-junta})}

The general proof strategy is the same as in the proof of Theorem~\ref{thm:main}. Namely, first we prove
the theorem in the special case where $k$ is a power of 2, and then generalise to arbitrary $k$. The latter generalisation is very similar to in Theorem~\ref{thm:main}. Hence, the main extra ingredient required is a refinement of Claim~\ref{claim:binary}.

Recall that Claim~\ref{claim:binary} asserts that if $k=2^s$ and $a_j^x$ is a fibre,
naturally identified with $[k]$, then
\begin{equation}\label{eq:restate}
\#\{(z,i) \in [k] \times [s]:\ z_i = 0,\ f_j^x(z) \neq f_j^x(z + 2^{i-1})\}
\leq (k-1) |\partial(f_j^x)|.
\end{equation}
Write $f = 1_{A}$, where $A \subset [k]^n$. Let us denote the restriction of $A$ to the fibre $a_j^x$ by $A_j^x$, and consider $A_j^x$ as a subset of $[k]$. Let
$$\phi: \{0,1\}^s \to [k];\quad (x_1,\ldots,x_s) \mapsto 1+\sum_{i=1}^{s}x_i 2^{i-1},$$
as in the proof of Theorem \ref{thm:main}. Then the LHS of (\ref{eq:restate}) is precisely $|\partial_{Q_s}(\phi^{-1}(A_j^x))|$, the size of the edge-boundary (in the discrete cube) of the set $\phi^{-1}(A_j^x) \subset \{0,1\}^s$. For brevity, we introduce the following notation. If $B \subset [k]$, define $\tilde{\partial}(B) = \partial_{Q_s}(\phi^{-1}(B))$, the edge-boundary (in the discrete cube) of the set $\phi^{-1}(B) \subset \{0,1\}^s$. Similarly, for each $i \in [s]$, we define $\tilde{\partial}_i(B)$ to be the set of direction-$i$ edges in $\tilde{\partial}(B)$, so that $\tilde{\partial}(B) = \dot\cup_{i=1}^s \tilde{\partial}_i(B)$. Then we have
$$|\tilde{\partial}_i(B)| = \#\{z \in [k]:\ z_i = 0,\ 1_B(z) \neq 1_B(z + 2^{i-1})\}\quad (i \in [s]),$$
and
$$|\tilde{\partial}(B)| = \#\{(z,i) \in [k] \times [s]:\ z_i = 0,\ 1_B(z) \neq 1_B(z + 2^{i-1})\}.$$
In this new notation, Claim~\ref{claim:binary} states that for all $j$ and $x$,
$$|\tilde{\partial}(A_j^x)| \leq (k-1) |\partial(A_j^x)|.$$
We will improve this inequality in two steps:
\begin{enumerate}
\item Obtain an improved upper bound on $|\tilde{\partial}I|$ for
any interval $I \subset [k]$, in terms of $|I|$.
\item Represent $A_j^x$ as a disjoint union of intervals in $[k]$, where the number
of intervals depends on $|\partial(A^x_j)|$, and combine the bounds corresponding to intervals into an upper bound for $|\tilde{\partial}(A_j^x)|$, using a convexity argument.
\end{enumerate}
The idea behind the proof of (1) is to generalize the proof strategy used in Proposition~2.3
of~\cite{keller} in the monotone case. If $A$ is monotone, then on each fibre $a_j^x$,
the set $A^x_j$ is of the special form $I=\{1,2,\ldots,\ell\}$ for some $\ell$. It was shown in~\cite{keller}
that for such an `initial interval' $I$, $|\tilde{\partial}I|$ can be easily
bounded from above in terms of $|I|$. In the general case considered here, the set $A_j^x$ does not
have such a nice structure, but it can be represented as a disjoint union of intervals
$I=\{\ell_1,\ell_1+1,\ldots,\ell_2-1,\ell_2\} \subset [k]$. We show that the argument of~\cite{keller}
applies (with a different constant) also for general intervals.

\begin{lemma}\label{lemma:refined}
Let $k=2^s$, for some $s \in \mathbb{N}$, and let $I \subset [k]$ be an interval with $0 < |I| \leq k/2$. Then
\[
|\tilde{\partial}(I)| \leq 6 |I| \log_2(k/|I|).
\]
\end{lemma}

\begin{proof}
Let $q \in \mathbb{Z}_{\geq 0}$ such that $2^q \leq |I| < 2^{q+1}$. We claim that
\[
|\tilde{\partial}_i(I)|  \leq \left\lbrace
  \begin{array}{c l}
    2^{i}, & \mbox{if }1 \leq i \leq q+1;\\
    2|I|, & \mbox{if }q+2 \leq i \leq s.
  \end{array}
\right.
\]
Indeed, since $I$ is an interval in $[k]$, for all $i$, the relation
$1_I(z) \neq 1_I(z + 2^{i-1})$ can be satisfied only for
$\min I - 2^{i-1} \leq z \leq \min I -1$ or $\max I - 2^{i-1}+1 \leq z \leq \max I$.
Hence, $|\tilde{\partial}_i(I)| \leq 2^{i}$. For $i>q+1$, the restriction is even stronger:
$z$ must satisfy either $z \in I$, or $\min I-2^{i-1} \leq z \leq \max I-2^{i-1}$. Thus, $|\tilde{\partial}_i(I)| \leq 2|I|$. Therefore,
\begin{align*}
|\tilde{\partial}(I)| & = \sum_{i=1}^s |\tilde{\partial}_i(I)|\\
& \leq \sum_{i=1}^{q+1} 2^{i} + \sum_{i=q+2}^{s} 2|I|\\
& \leq 2^{q+2} + 2|I|(s-q-1)\\
& \leq 4|I|+2|I|\log_2 \frac{k}{|I|}\\
& \leq 6|I|\log_2(k/|I|).
\end{align*}
\end{proof}

\begin{proposition}\label{prop:refined-binary}
Let $k=2^s$. Let $a_j^x$ be a fibre, naturally identified with $[k]$, and let
$\ell: = |A_j^x|$ and $m:=|\partial(A_j^x)|$. Suppose that $0 < \ell \leq k/2$. Then
\begin{equation}\label{Eq:Refined1}
|\tilde{\partial}(A_j^x)| \leq 6 \ell \log_2(mk/\ell).
\end{equation}
\end{proposition}
\begin{proof}[Proof of Proposition~\ref{prop:refined-binary}:]
Since $|\partial(A_j^x)|=m$, it follows that $A^x_j$ is a disjoint union of $\lceil m/2 \rceil$ or
$\lceil m/2 \rceil+1$ intervals of consecutive integers. Denote these intervals by $I_1,I_2,\ldots,I_{m'}$.
By Lemma~\ref{lemma:refined}, for each interval $I_i$, we have
\[
|\tilde{\partial}(I_i)| \leq 6 |I_i| \log_2(k/|I_i|).
\]
Clearly, we have
\[
|\tilde{\partial}(A_j^x)| = |\tilde{\partial}(\bigcup_{i=1}^{m'}I_i)| \leq \sum_{i=1}^{m'} |\tilde{\partial}(I_i)|.
\]
Substituting in the assertion of the lemma, we get
\[
|\tilde{\partial}(A_j^x)| \leq 6 \sum_{i=1}^{m'} |I_i| \log_2(k/|I_i|).
\]
Since the function $\varphi(x)=x \log_2(k/x)$ is concave in $(0,k)$, we have
\[
|\tilde{\partial}(A_j^x)| \leq 6m' \frac{\sum_{i=1}^{m'} |I_i|}{m'}
\log_2 \left(\frac{km'}{\sum_{i=1}^{m'} |I_i|} \right) = 6 \ell \log_2(km'/\ell) \leq 6\ell \log_2(km/\ell),
\]
as required.
\end{proof}

\begin{proof}[Proof of Theorem \ref{thm:refined-junta}]
Let $A \subset [k]^n$ and let $f = 1_A$. As before, define $\tilde{f}:= f \circ \phi_{(n)}:\ \{0,1\}^{sn} \to \{0,1\}$.

Proposition \ref{prop:refined-binary} immediately implies that for any $j \in [n]$ and any $x \in [k]^n$ such that $0 < \ell = |A_j^x| \leq k/2$ and $m= |\partial( A_j^x)|$, we have
$$\sum_{i=1}^{s} \Inf_i(\phi^{-1}(A_j^x)) \leq 2^{-s+1} \cdot 6 \ell \log_2(km/\ell) = 12(\ell/k) \log_2(mk/\ell),$$
so certainly
\begin{equation}\label{eq:sym} \sum_{i=1}^{s} \Inf_i(\phi^{-1}(A_j^x)) \leq 24(\ell/k)(1-\ell/k)\log_2\left(\frac{m}{(\ell/k)(1-\ell/k)}\right).\end{equation}
Since this last inequality is symmetric under $\ell \mapsto k - \ell$, it holds for all $\ell \in [k-1]$, and so for all $j$ and $x$ such that $f_j^x$ is non-constant.

By averaging (\ref{eq:sym}) over all $x \in [k]^{n-1}$ (for a fixed $j$), and then summing over
all $j \in [n]$, we obtain:
\begin{align*}\sum_{i=1}^{sn} \Inf_i(\tilde{f}) & \leq 24\sum_{j=1}^n \mathbb{E}_{x \in [k]^{n-1}} \Var(f_j^x) \log_2{ \frac{|\partial f_j^x|}{\Var(f_j^x)}} 1\{f_j^x \textrm{ is non-constant}\}\\
& = 24\sum_{j=1}^n \mathbb{E}_{x \in [k]^{n-1}}h^*(f_j^x).\end{align*}
(Note that on each fibre where $f_j^x$ is non-constant, we have
$$\Var(f_j^x) \log_2{ \frac{|\partial f_j^x|}{\Var(f_j^x)}} = (\ell/k)(1-\ell/k)\log_2\left(\frac{m}{(\ell/k)(1-\ell/k)}\right),$$
in the notation of Proposition~\ref{prop:refined-binary}.)

Let $\epsilon >0$. By Theorem \ref{thm:friedgut-junta}, it follows that for every $\epsilon >0$, $\tilde{f}$ is $\epsilon$-close to some $M$-junta $\tilde{g}:\{0,1\}^{sn} \to \{0,1\}$, where
$$M \leq \exp \left( \frac{24C_0}{\epsilon} \sum_{j=1}^n \mathbb{E}_{x \in [k]^{n-1}} h^*(f_j^x)\right),$$
and $C_0$ is the constant of Theorem \ref{thm:friedgut-junta}. Let $g = \tilde{g} \circ (\phi_{(n)})^{-1}$. Then $f$ is $\epsilon$-close to $g:[k]^n \to \{0,1\}$, and $g$ is also an $M$-junta. This proves
Theorem~\ref{thm:refined-junta} in the case where $k$ is a power of 2.

The generalisation to arbitrary $k$ is similar to in the proof of Theorem~\ref{thm:main}, with one extra complication. Suppose $k \geq 3$. As before, let $l>k$ be a power of $2$ such that
\begin{equation}\label{eq:sizecond} \left(1+\frac{k}{l-k}\right)^n \leq 2,\end{equation}
and partition $[l]^n$ into $k^n$ large, rectangular blocks $(R_x:\ x \in [k]^n)$, where
$$R_x := \{y \in [l]^n:\ (\lceil y_1k/l\rceil, \lceil y_2k/l\rceil,\ldots,\lceil y_nk/l\rceil) = x\}.$$
Recall that
$$\lfloor l/k \rfloor^n \leq |R_x| \leq \lceil l/k \rceil^n \quad \forall x \in [k]^n.$$
Let $A \subset [k]^n$. As before, define
$$\breve{A} = \cup_{x \in A} R_x \subset [l]^n.$$
Let $\breve{f} = 1_{\breve{A}}$, and let $\breve{f}_j^x:[l] \to \{0,1\}$ denote the restriction of $\breve{f}$ to the fibre $a_j^x$.

In the proof of Theorem \ref{thm:main}, it was easy to see that $|\partial \breve{A}|/l^{n-1} \leq 2 |\partial A|/k^{n-1}$. In this case, we need the following analogous claim.
\begin{claim}
\label{claim:est}
For all $j \in [n]$, we have
$$\mathbb{E}_{y \in [l]^{n-1}} h^*(\breve{f}_j^y) \leq \frac{9}{2} \mathbb{E}_{x \in [k]^{n-1}} h^*(f_j^x).$$
\end{claim}
\begin{proof}[Proof of Claim \ref{claim:est}.]
Let $j \in [n]$. Observe that each direction-$j$ fibre in $[l]^n$ arises from exactly one direction-$j$ fibre in $[k]^n$, and that each direction-$j$ fibre in $[k]^n$ has between $\lfloor l/k \rfloor^{n-1}$ and $\lceil l/k \rceil^{n-1}$ direction-$j$ fibres arising from it. Let $y \in [l]^n$, and let $x \in [k]^n$ such that $y \in R_x$; then the fibre $a_j^y$ arises from the fibre $a_j^x$. We assert that
\begin{equation} \label{eq:pointwise} h^*(\breve{f}_j^y) \leq \tfrac{9}{4} h^*(f_j^x).\end{equation}
Indeed, if $f_j^x$ is constant, then so is $\breve{f}_j^y$, so $h^*(\breve{f}_j^y) = \tfrac{3}{2} h^*(f_j^x) = 0$ and we are done. Hence, we may assume that $f_j^x$ is non-constant. Observe that
$$|\tilde{A}_j^y| \leq \lceil l/k \rceil |A_j^x| < (l/k+1) |A_j^x|.$$
Hence,
$$\Pr[\breve{f}_j^y=1] \leq (1+k/l) \Pr[f_j^x=1] \leq \tfrac{3}{2} \Pr[f_j^x=1],$$
where $\Pr$ denotes the uniform measure on the relevant fibre. Similarly,
$$1-\Pr[\breve{f}_j^y=1] = \Pr[\breve{f}_j^y=0] \leq \tfrac{3}{2} \Pr[f_j^x=0] = \tfrac{3}{2} (1-\Pr[f_j^x=1]).$$
It follows that
\begin{align*} \Var(\breve{f}_j^y) & = \Pr[\breve{f}_j^y=1] (1-\Pr[\breve{f}_j^y=1])\\
& \leq \tfrac{9}{4}\Pr[f_j^x=1](1-\Pr[f_j^x=1])\\
& = \tfrac{9}{4} \Var(f_j^x).\end{align*}
Hence,
\begin{align*} h^*(\breve{f}_j^y) & = \Var(\breve{f}_j^y) \log_2{ \frac{|\partial \breve{f}_j^y|}{\Var(\breve{f}_j^y)}}\\
& \leq \tfrac{9}{4} \Var(f_j^x) \log_2{ \frac{|\partial \breve{f}_j^y|}{\Var(f_j^x)}}\\
& = \tfrac{9}{4} \Var(f_j^x) \log_2{ \frac{|\partial f_j^x|}{\Var(f_j^x)}}\\
& = \tfrac{9}{4} h^*(f_j^x),\end{align*}
using the facts that $|\partial \breve{f}_j^y| = |\partial f_j^x|$, that $\Var(f_j^x) \leq 1/4$, and that the function $t \mapsto t \log_2 (1/t)$ is concave and strictly increasing on $(0,1/4)$. This proves (\ref{eq:pointwise}).

It follows that
\begin{align*}
\sum_{y \in [l]^{n-1}} h^*(\breve{f}_j^y)& \leq \tfrac{9}{4} (\lceil l/k \rceil)^{n-1} \sum_{x \in [k]^{n-1}} h^*(f_j^x)\\
& \leq \tfrac{9}{4} (l/k +1)^{n-1} \sum_{x \in [k]^{n-1}} h^*(f_j^x).\end{align*}
Hence,
\begin{align*} \mathbb{E}_{y \in [l]^{n-1}} h^*(\breve{f}_j^y) & \leq \tfrac{9}{4} (1+k/l)^{n-1} \mathbb{E}_{x \in [k]^{n-1}} h^*(f_j^x)\\
& \leq \tfrac{9}{2} \mathbb{E}_{x \in [k]^{n-1}} h^*(f_j^x),\end{align*}
using (\ref{eq:sizecond}). This proves the claim.
\end{proof}
The rest of the argument is exactly as in proof of Theorem \ref{thm:main}, so is omitted here. This completes the proof of our refined junta theorem.
\end{proof}




\section{Conclusions and Open Problems}
\label{sec:conclusion}
Theorem \ref{thm:main} describes the structure of {\em large} subsets of the grid, whose edge-boundary has size within a constant factor of the minimum. However, when $|A|/k^n \leq 1/n$, its conclusion is vacuous for all $\epsilon$. Indeed, observe that the conclusion is trivial if $\epsilon \geq |A|/k^n$, as in this case we can take $g \equiv 0$. However, Theorem \ref{thm:grid-iso} implies that for any $A \subset [k]^n$ with $|A|/k^n \leq 1/2$,
$$|\partial A| \geq e \frac{|A|}{k} \ln(k^n/|A|).$$ Therefore,
the bound on the junta-size in Theorem \ref{thm:main} is
\begin{align*} \exp\left(C_0|\partial A|/(k^{n-1}\epsilon) \right) & \geq \exp\left(eC_0(|A|/(k^{n}\epsilon)) \ln (k^n/|A|)\right)\\
& \geq \exp(eC_0 \ln (k^n/|A|))\\
& \geq \exp(eC_0 \ln n)\\
& = n^{eC_0}\\
& \geq n,\end{align*}
when $\epsilon \leq |A|/k^n \leq 1/n$.

It would be interesting to obtain a description of {\em small} subsets of the grid, whose edge-boundary has size within a constant factor of the minimum. In the case of the discrete cube, Kahn and Kalai \cite{kk} made the following conjecture.

\begin{conjecture}[Kahn and Kalai]
\label{conj:kk}
For any \(L > 0\), there exist \(L' > 0\) and \(\delta >0\) such that the following holds. If \(A \subset \{0,1\}^{n}\) is monotone increasing, with measure $\alpha = \frac{|A|}{2^{n}} \leq 1/2$, and with edge-boundary satisfying
$$|\partial A| \leq L |A| \log_2(2^n/|A|),$$
then there exists a subcube \(C \subset \{0,1\}^n\) with measure at least $\alpha^{L'}$ and all fixed coordinates equal to 1, such that
 $$\frac{|A \cap C|}{|C|} \geq (1+\delta)\alpha.$$
\end{conjecture}

Note that if we write
$$\Phi_n(t) = \min\{|\partial A|:\ A \subset \{0,1\}^n,\ |A|=t\},$$
then
$$ t \log_2(2^n/t) \leq \Phi_n(t) \leq 2 t\log_2 (2^n/t) \ \forall t \leq 2^{n-1}.$$
(This follows easily from the edge-isoperimetric inequality of Harper, Lindsay, Bernstein and Hart; see for example \cite{hart}.) Hence, the hypothesis of the above conjecture is indeed that the edge-boundary is within a constant factor of the minimum.

The conclusion of the conjecture says that there is a fairly large subcube such that the restriction of $A$ to this subcube has measure exceeding that of $A$ by a constant factor. This conclusion is, of course, much weaker than being able to approximate $1_{A}$ by a junta depending upon a bounded number of coordinates, as in Theorem \ref{thm:friedgut-junta}. However, the following `tribes' construction shows that the latter cannot always be achieved. Let $\eta =2^{-r}$ for some $r \in \mathbb{N}$. Let $n = 2^{2^{s-r}}$ for some $s > r$. Let $p = 2^s = \frac{1}{\eta} \log_2 n$, and let $q = n/p$. (Note that $q \in \mathbb{N}$, as $n$ and $p$ are both powers of 2.) Partition $[n]$ into $q$ `tribes' $T_i$, each of size $p$. Let
$$A = \{x \in \{0,1\}^n:\ x_j= 1\ \forall j \in T_i, \textrm{ for some tribe }T_i\}.$$
Then it is easily checked that provided $n$ is sufficiently large depending on $\eta$, we have
$$|\partial A| < \frac{1}{1-\eta} |A|\log_2(2^n/|A|),$$
but $1_{A}$ is not $\tfrac{1}{3}$-close to any junta depending upon at most $(1-\tfrac{e}{3}-\tfrac{1}{12})n$ coordinates.

In \cite{almostiso}, the second author made the analogous conjecture for arbitrary (i.e, not necessarily monotone) subsets.
\begin{conjecture}
For any \(L > 0\), there exist \(L' > 0\) and \(\delta >0\) such that the following holds. If \(A \subset \{0,1\}^{n}\) has measure $\alpha = \frac{|A|}{2^{n}} \leq 1/2$, and edge-boundary satisfying
$$|\partial A| \leq L |A| \log_2(2^n/|A|),$$
then there exists a subcube \(C \subset \{0,1\}^n\) with measure at least $\alpha^{L'}$, such that
\[\frac{|A \cap C|}{|C|} \geq (1+\delta)\alpha.\]
\end{conjecture}

Here, we make the analogous conjecture for subsets of the $\ell^1$-grid.

\begin{conjecture}
\label{conj:grid}
For any \(L > 0\), there exist \(L' > 0\) and \(\delta >0\) such that the following holds. If \(A \subset [k]^{n}\) has measure $\alpha = \frac{|A|}{k^{n}} \leq 1/2$, and edge-boundary satisfying
$$|\partial A| \leq L \min\{|A|^{1-1/r} rk^{n/r-1}:\ r \in \{1,2,\ldots,n\}\},$$
then there exists a cartesian product set \(C = S_1 \times S_2 \times \ldots \times S_n \subset [k]^n\) with measure at least $\alpha^{L'}$, such that
\[\frac{|A \cap C|}{|C|} \geq (1+\delta)\alpha.\]
\end{conjecture}
We remark that one cannot demand that the cartesian product set in Conjecture \ref{conj:grid} be a cuboid (i.e. that $|S_i| = a$ or $k$ for all $i$, for some $a \in [k]$.)

It would also be of interest to improve the bound on the junta-size in Theorem \ref{thm:sum-inf}. Our construction in Section 2 shows that the dependence upon $k$ is optimal up to the absolute constant factor in the exponent, but we conjecture that the dependence upon $l$ could be removed entirely:
\begin{conjecture}
Suppose $f:\mathbb{Z}_k^n \to \mathbb{Z}_l$ with
$$\sum_{j=1}^{n} \mathbb{E}_x [|f(x)-f(x\oplus e_j)|'] \leq B.$$
Then $f$ is $\epsilon$-close to some $M$-junta $g: \mathbb{Z}_k^n \to \mathbb{Z}_l$, where
 $$M \leq \exp(C_4 B k /\epsilon),$$
 for some absolute constant $C_4$.
\end{conjecture}
Of course, if this were true, it would lead to a strengthening of Theorem \ref{thm:lipschitz-tori} with no dependence upon $l$.
Likewise, we are not sure that the exponential dependence on $\alpha$ in Theorem~\ref{thm:lipschitz-tori} is optimal. In fact,
in all examples we considered the dependence is linear, and it may well be that this is really the case. We make the following conjecture in this regard. 
\begin{conjecture}
\label{conj:lipschitz-tori}
Suppose $f = (f_1,\ldots,f_m): \mathbb{Z}_k^{n} \to \mathbb{Z}_l^{m}$ is $\alpha$-Lipschitz with respect to the $L^1$-norm. Then for any $\delta,\epsilon >0$, there are at least $(1-\delta)m$ coordinates $i \in [m]$ such that $f_i$ is $\epsilon$-close to some $M$-junta $g_i: \mathbb{Z}_k^n \to \mathbb{Z}_l$, where
$$M \leq \alpha \exp(Ck / (\delta \epsilon)),$$
and $C$ is an absolute constant. 
\end{conjecture} 

\subsection*{Ackowledgement}
The authors would like to thank Igor Shinkar for helpful discussions, and an anonymous referee for pointing out a mistake (in Example 2) in an earlier version of the paper.

\end{document}